\documentclass[reqno]{amsart}
\usepackage{hyperref}
\usepackage{amssymb}

\usepackage{graphicx}

\theoremstyle{plain}
\newtheorem{theorem}{Theorem}[section]

\newtheorem{lemma}[theorem]{Lemma}

\theoremstyle{definition}

\theoremstyle{remark}
\newtheorem{remark}[theorem]{Remark}
\numberwithin{equation}{section}

\begin{document}

\title[Nonexistence of entire solutions of semilinear elliptic systems]
{Nonexistence of nonnegative entire solutions of semilinear elliptic systems}

\author[A. Gladkov]{Alexander Gladkov}
\address{Alexander Gladkov \\ Department of Mechanics and Mathematics
\\ Belarusian State University \\  4  Nezavisimosti Avenue \\ 220030
Minsk, Belarus  and  Peoples' Friendship University of Russia (RUDN University) \\  6 Miklukho-Maklaya street \\  117198 Moscow,  Russian Federation}    \email{gladkoval@mail.ru }

\author[S. Sergeenko]{Sergey Sergeenko}
\address{Sergey Sergeenko \\ Department of Mathematics and Information Technologies, Vitebsk State University named after P M Masherov,
33 Moskovskii pr., Vitebsk, 210038, Belarus}
\email{sergeenko@vsu.by}

\subjclass[2010]{35J47, 35J91}
\keywords{Semilinear elliptic system; entire solutions; nonexistence}

\begin{abstract}
We consider the second order semilinear elliptic system
 $\Delta u= p\left(  x\right) v^\alpha,$ $\Delta v= q\left(x\right) u^\beta,$ where
  $x \in \mathbf{R}^N,$ $N \geq 3,$ $\alpha$ and $\beta$ are positive constants,
 $p$  and  $q$ are nonnegative continuous functions. We prove that nontrivial nonnegative  entire solutions
 fail to exist if the functions  $p$  and  $q$ are of slow decay.
  \end{abstract}

\maketitle


\section{Introduction}

We consider the second order semilinear elliptic system of the following form
\begin{equation}\label{eq:main:task}
  \left\{\begin{aligned}
      \Delta u &= p\left( x\right) v^\alpha, \\
      \Delta v&= q\left( x\right) u^\beta,
    \end{aligned}\right.
\end{equation}
where $x\in \mathbf{R}^N,$ $N \geq 3,$ $\alpha,$ $\beta$ are positive numbers,
$p$ and $q$ are nonnegative continuous functions defined on $ \mathbf{R}^N.$
Our objective is to establish conditions for the nonexistence of nontrivial nonnegative entire solutions of
(\ref{eq:main:task}).  An entire solution of (\ref{eq:main:task})  is defined to be
a vector-valued function $\left(u, v\right) \in C^2 \left(\mathbf{R}^N\right) \times C^2
\left(\mathbf{R}^N\right)$ which satisfies  (\ref{eq:main:task}) at every point in $R^N.$

To formulate main result of this paper we introduce the functions $\tilde p (r)$ and $\tilde q (r)$ by
  \begin{equation*}\label{1}
  \tilde p (r) = \begin{cases}
 \left( \frac{1}{\omega_N r^{N-1}} \int_{|x| =r} p^{1/(1-\alpha)} (x) \, dS \right)^{1-\alpha}, \; \alpha >1, \\
  \min_{|x| = r} p(x),\; \alpha =1,
    \end{cases}\end{equation*}
\begin{equation*}\label{2}
  \tilde q (r) = \begin{cases}
 \left( \frac{1}{\omega_N r^{N-1}} \int_{|x| =r} q^{1/(1-\beta)} (x) \, dS \right)^{1 - \beta}, \; \beta >1, \\
  \min_{|x| = r} q(x),\; \beta =1,
    \end{cases}\end{equation*}
where $r>0,\,$ $\omega_N$ is the surface area of the unit sphere in $\mathbf{R}^N.$
We set $\tilde p (r) = 0$ and $\tilde q (r) = 0$ if
\begin{equation*}
  \int_{|x| =r}  p^{1/(1-\alpha)} (x) \, dS  = \infty  \,\,\, \textrm{and} \,\,\,  \int_{|x| =r} q^{1/(1-\beta)} (x) \, dS  = \infty,
\end{equation*}
respectively. We note that $ p (x) = \tilde p (|x|), \, $  $ q (x) = \tilde q (|x|)$
when $ p $ and $ q $ are spherically symmetric functions (i.e.,  $p (x)= p (|x|)$  and $q (x)= q (|x|)$).
Suppose that $ \tilde p $ and $ \tilde q $ satisfy
\begin{equation}\label{3}
\tilde p(r) \geq \frac{L_1}{r^\lambda \ln^\nu r} ,\quad
\tilde q(r) \geq \frac{L_2}{r^\mu \ln^\xi r}, \quad r \geq r_0 > 1,
\end{equation}
where $L_1 >0, $ $L_2 >0$ and  $\lambda,$ $\nu,$ $\mu,$ $\xi$ are constants. Our main result is as follows.
\begin{theorem}\label{Th1}
Let $N \geq 3,$ $\alpha \geq 1,$ $\beta \geq 1,$ $\alpha \beta > 1$ and $p$ and $q$ satisfy (\ref{3}). If at least one from the following conditions holds
     {\rm
  \begin{equation*}
 \begin{split}
 & \textrm{(i)} \,\,\,\,\,    2 - \mu+  \beta (2 - \lambda)  > 0; \\
 & \textrm{(ii)} \,\,\,\,     2 - \lambda + \alpha (2 - \mu) > 0; \\
 & \textrm{(iii)} \,\,\,      2 - \mu+  \beta (2 - \lambda) = 0, \; \lambda < 2, \; 1 - \xi - \beta \nu > 0;  \\
 & \textrm{(iv)} \,\,\,       2 - \lambda + \alpha (2 - \mu) = 0, \; \mu < 2, \; 1 - \nu - \alpha \xi > 0;   \\
 & \textrm{(v)} \,\,\,\,      \lambda = 2, \; \mu = 2, \; 1 - \xi + \beta (1-\nu) > 0; \\
& \textrm{(vi)} \,\,\,        \lambda = 2, \; \mu = 2, \;  1-\nu+\alpha(1-\xi)>0,
 \end{split}
\end{equation*}
}
then there are not nontrivial nonnegative entire solutions of (\ref{eq:main:task}).
\end{theorem}

The problem of the existence and nonexistence of entire solutions of scalar elliptic equations
has been investigated by many authors (see e.g. \cite{NU, GS1, Yang, GS2, YZ, F, FS, DGGW, K1, KK1, KK2} and the references therein).
 Entire solutions for semilinear elliptic systems have been considered in many papers also (see, for example, previous works
 \cite{KK, Y, T1, LW, CR, TU, T3, ZSX, L, Z, C}). In particular, Theorem~\ref{Th1} was proved in \cite{Y, T1} for $\nu = \xi = 0.$
Lair in \cite{L} established the existence of positive entire  solutions to the elliptic system (\ref{eq:main:task}) with spherically symmetric coefficients
$ p  $ and  $ q.$ More precisely, the system (\ref{eq:main:task}) has positive entire solutions if $\alpha \beta > 1$ and
$ p (|x|)$ and $q (|x|)$ satisfy at least one of the following conditions
\begin{align}
  \label{Lair-2011:exist-Q}
  \int_0^\infty t p(t) \left( t^{2-N} \int_0^t s^{N-3} \int_0^s \tau
    q(\tau) d\tau ds \right)^\alpha dt &<\infty,
  \\
  \label{Lair-2011:exist-P}
  \int_0^\infty t q(t) \left( t^{2-N} \int_0^t s^{N-3} \int_0^s \tau
    p(\tau) d\tau ds \right)^\beta dt &<\infty.
\end{align}

This paper is organized as follows. Some auxiliary propositions are proved in Section 2. Section 3 is devoted to the proof of Theorem~\ref{Th1}.
In Section 4, we discuss the optimality of obtained results.


\section{Auxiliary propositions}\label{class}

Let $P$ and $Q$ be nonnegative nonincreasing continuous functions  and $h$ and $g$ be nonnegative continuous functions.
We introduce the functions $y$ and $z$ in the following way
\begin{equation*}
    y(r) = \int_R^r (r-s)^m P(s) g^\alpha(s) \, ds,\quad
    z(r) = \int_R^r (r-s)^n Q(s) h^\beta(s) \, ds,
  \end{equation*}
where $r > R >0, \,$ $m,n \in \mathbf{N}.$

\begin{lemma}\label{l:I1}
Let $m,n \in \mathbf{N},$ $b > 1,$ $\alpha > 0,$ $\beta > 0,$  $\alpha \beta > 1$ and for $r \in [R, bR]$
\begin{equation*}
 \begin{split}
  h(r) &\geq \int_R^r (r-s)^m P(s) g^\alpha(s) \, ds,
    \\
    g(r) &\geq \int_R^r (r-s)^n Q(s) h^\beta(s) \, ds.
 \end{split}
\end{equation*}
 Then
  \begin{equation}  \label{ineq:help:eysupl}
    y^{\frac{\alpha \beta - 1}{(n+1)\alpha+m+1}}(A)
    \int_A^{bR} \left( P(r) Q^\alpha(r) \right)^{\frac{1}{(n+1)\alpha + m +1}}  \, dr \leq C_y,
    \end{equation}
  \begin{equation}  \label{ineq:help:ezsupl}
    z^{\frac{\alpha \beta -1}{(m+1)\beta+n+1}}(A)
    \int_A^{bR} \left( Q(r) P^\beta(r) \right)^{\frac{1}{(m+1)\beta +n +1}} \, dr \leq C_z,
    \end{equation}
 where $A \in (R,bR),$ $C_y$ and $C_z$ are positive constants which do not depend on $R.$
\end{lemma}
\begin{proof}
We prove only (\ref{ineq:help:ezsupl}), the proof of (\ref{ineq:help:eysupl}) is similar. Note that
\begin{equation}\label{ineq:help:ugeqy-vgeqz}
  h(r) \geq y(r) \geq 0,\quad
  g(r) \geq z(r) \geq 0,
\end{equation}
\begin{equation}\label{24}
y^{(i)}(r)\geq 0, \, y^{(i)}(R) = 0  \,\,\, \textrm{for} \,\,\,  i \in \{1,2,\dots,m\},
\end{equation}
\begin{equation}\label{25}
z^{(i)}(r)\geq 0, \, z^{(i)}(R) = 0  \,\,\, \textrm{for} \,\,\,  i \in \{1,2,\dots,n\},
\end{equation}
\begin{equation}\label{der}
  y^{(m+1)}(r) = m! P(r) g^\alpha(r) \geq 0, \quad   z^{(n+1)}(r) = n! Q(r) h^\beta(r)\geq 0.
\end{equation}
  Thus, from (\ref{ineq:help:ugeqy-vgeqz}), (\ref{der}), we have
\begin{align}
  y^{(m+1)}(r) &\geq m! P(r)
  z^\alpha(r), \label{ineq:help:ydiff}
  \\
  z^{(n+1)}(r) &\geq n! Q(r)
  y^\beta(r). \label{ineq:help:zdiff}
\end{align}
Now we multiply (\ref{ineq:help:ydiff}) by $z'(r)$ and integrate over $[R,r].$
Using the integration by parts, (\ref{24}), (\ref{25}) and the monotonicity of $P,$  we obtain
\begin{equation*}
  y^{(m)}(r) z'(r) \geq \frac{m!}{\alpha+1} P(r) (z(r))^{\alpha+1}.
\end{equation*}
Repeating this process appropriately many times, we get
\begin{equation}\label{ineq:help:ydiffzdiff}
  y(r) (z'(r))^{m+1} \geq \frac{m!}{(\alpha+1) (\alpha+2)\dots (\alpha+m+1)} P(r) (z(r))^{\alpha+m+1}.
\end{equation}
 A combination of (\ref{ineq:help:zdiff}) and (\ref{ineq:help:ydiffzdiff}) gives
\begin{equation}\label{ineq:help:zdiffdeg}
  z^{(n+1)}(r) (z'(r))^{(m+1)\beta} \geq c_1 Q(r) P^\beta(r)
  (z(r))^{(\alpha + m + 1)\beta},
\end{equation}
where a positive constant $c_1$ does not depend on $R.$ From now on, without causing any confusion, we may
use $c_i,\,$ $C_i,$ $\bar C_i,$ $\hat C_i,$ $\breve C_i$ or $\tilde C_i \,$ $(i =0, 1, 2, ...)$ to denote various positive constants.
Applying $n$ times the operation of multiplying (\ref{ineq:help:zdiffdeg})
by $z'(r)$ and integration over $ [R, r], $ we obtain
\begin{equation}\label{ineq:dzpower}
  (z'(r))^{(m+1)\beta+n+1} \geq c_2  Q(r) P^\beta(r)  (z(r))^{(\alpha+m+1)\beta+n}.
\end{equation}

Without loss of generality, we can suppose that $z(A) >0.$
From (\ref{ineq:dzpower}), we find
\begin{equation*}
  (z(r))^{- \frac{(\alpha + m + 1) \beta + n}{(m + 1) \beta + n +
      1}} z'(r) \geq \left(c_2  Q(r) P^\beta(r)  \right)^
  {\frac{1}{(m + 1) \beta + n + 1}}, \quad r \in [A, bR].
\end{equation*}
Integrating this relation from $A$ to $bR,$ then leads to the inequality
\begin{gather*}
  C_z \left(z^{-\frac{\alpha\beta-1}{(m+1)\beta+n+1}}(A)-
    z^{-\frac{\alpha\beta-1}{(m+1)\beta+n+1}}(bR)\right) \geq
  \int_A^{bR} \left(  Q(r) P^\beta(r)
  \right)^{\frac{1}{(m+1)\beta+n+1}} \, dr,
\end{gather*}
which completes the proof.
\end{proof}

Let us introduce the functions $u_k(r),$ $v_k(r)$  by the following recurrence relationships
\begin{align}
  u_0(r) &= 1,   \quad    v_0(r) = 1, \label{eq:main3:defu0v0}
  \\
  u_k(r) &= \frac{1}{N-2} \int_\rho^r s \tilde p(s) \left[ 1 -
    \left(\frac{s}{r}\right)^{N-2} \right] v_{k-1}^\alpha(s) \, ds,  \quad k \in \mathbb{N}, \label{eq:main3:defuk}
  \\
  v_k(r) &= \frac{1}{N-2} \int_\rho^r s \tilde q(s) \left[ 1 -
    \left(\frac{s}{r}\right)^{N-2} \right] u_{k-1}^\beta(s) \,
  ds,  \quad k \in \mathbb{N} \label{eq:main3:defvk}
\end{align}
and define the functions
\begin{equation*}\label{yz}
  y_k(R) = \int_R^{aR} (aR-s)^m s^{1-m} \tilde p(s) v_k^\alpha(s) \, ds,\quad
  z_k(R) = \int_R^{aR} (aR-s)^n s^{1-n} \tilde q(s) u_k^\beta(s) \, ds,
\end{equation*}
where $\rho > 0, \, R > 1, \, a > 1.$

Now we prove an auxiliary statement which has independent interest.
\begin{theorem}\label{Th2}
Let $N \geq 3,$ $\alpha \geq 1,$ $\beta \geq 1,$ $\alpha \beta > 1,$  $r^{1-m} \tilde p (r)$ and $r^{1-n} \tilde q (r)$
be nonincreasing functions on $(\rho, \infty)$ for some $m,n \in \mathbf{N}$ and $\rho > 0.$
If at least one from the following conditions holds
\begin{equation}\label{eq:main3:oscy}
   \limsup_{R \to \infty} y_k^{\frac{\alpha\beta -
      1}{(n+1)\alpha+m+1}}(R)
  \int_{aR}^{bR} \left( s^{1-m} \tilde p (s) (s^{1-n} \tilde q(s))^\alpha
  \right)^{\frac{1}{(n+1)\alpha+m+1}} \, ds=\infty
\end{equation}
or
\begin{equation}  \label{eq:main3:oscz}
  \limsup_{R \to \infty}
  z_k^{\frac{\alpha\beta-1}{(m+1)\beta+n+1}}(R)
  \int_{aR}^{bR} \left( s^{1-n} \tilde q(s) (s^{1-m} \tilde p (s))^\beta
  \right)^{\frac{1}{(m+1)\beta+n+1}} \, ds=\infty
\end{equation}
for some  $k \in \mathbf{N}$ and $a, \, b$ such that $1 < a < b \leq 2,$
then there are not nontrivial nonnegative entire solutions of (\ref{eq:main:task}).
\end{theorem}

\begin{proof}
Assume to the contrary that (\ref{eq:main:task}) has a nontrivial nonnegative entire solution $(u,v).$
Let $\overline{u} (r),$ $\overline{v} (r)$ denote the averages of $u (x),$ $v (x)$
over the sphere $|x| = r,$ respectively, that is,
\begin{equation*}  \label{aver}
 \overline{u} (r) = \frac{1}{\omega_N r^{N-1}} \int_{|x| =r} u (x) \, dS, \quad
  \overline{v} (r) = \frac{1}{\omega_N r^{N-1}} \int_{|x| =r} v (x) \, dS.
\end{equation*}
Then we can proceed analogously as in \cite{NS, Ni}  to find that
\begin{equation}  \label{4}
(r^{N-1} \overline{u}'(r))' \geq  r^{N-1} \tilde p (r) \overline{v}^\alpha (r), \, r>0, \quad \overline{u}'(0) = 0,
\end{equation}
\begin{equation}  \label{5}
(r^{N-1} \overline{v}'(r))' \geq  r^{N-1} \tilde q (r) \overline{u}^\beta (r), \, r>0, \quad \overline{v}'(0) = 0.
\end{equation}
Obviously, $\overline{u}'(r) \geq 0,$  $\overline{v}'(r) \geq 0, $ $r \geq 0$ and
$\overline{u}(r) > 0,$  $\overline{v}(r) > 0,$ $r > \rho$ for some $\rho > 0.$
Integrating (\ref{4}) and (\ref{5}) twice over $[0,r],$ we have for $r > \rho$
\begin{align}
 \overline{u}(r) &\geq  \overline{u}(0) + \frac{1}{N-2} \int_0^r s \tilde p(s) \left[1 -
    \left(\frac{s}{r}\right)^{N-2} \right] \overline{v}^\alpha(s) \, ds \nonumber \\
      &\geq \frac{1}{N-2} \int_\rho^r s \tilde p(s) \left[1 -
    \left(\frac{s}{r}\right)^{N-2} \right] \overline{v}^\alpha(s) \, ds  \label{eq:main3:radu}
\end{align}
and
\begin{align}
 \overline{v}(r) &\geq  \overline{v}(0) + \frac{1}{N-2} \int_0^r s \tilde q(s) \left[1 -
    \left(\frac{s}{r}\right)^{N-2} \right] \overline{u}^\beta(s) \, ds  \nonumber \\
       &\geq \frac{1}{N-2} \int_\rho^r s \tilde q(s) \left[1 -
    \left(\frac{s}{r}\right)^{N-2} \right] \overline{u}^\beta(s) \, ds.  \label{eq:main3:radv}
\end{align}
Now the principle of mathematical induction is used to prove the estimates
\begin{equation}  \label{6}
   \overline{u}(r) \geq C_{2k} u_k (r), \;   \overline{v}(r) \geq C_{2k+1} v_k(r),  \; r > \rho,
\end{equation}
where $u_k(r)$ and $v_k(r)$ are defined in (\ref{eq:main3:defu0v0}) --
(\ref{eq:main3:defvk}), $ k =0,1,2,\dots.$ It is easy to see that (\ref{6}) follows from
(\ref{eq:main3:defu0v0}) for $k=0.$  Assume (\ref{6}) is true for $k=l-1.$ In view of
(\ref{eq:main3:defuk}), (\ref{eq:main3:defvk}),  (\ref{eq:main3:radu}),  (\ref{eq:main3:radv}) and (\ref{6}) with $k=l-1,$ we have
 \begin{equation*}
   \overline{u}(r) \geq \frac{C^\alpha_{2l-1}}{N-2} \int_\rho^r s \tilde p(s) \left[1 -
    \left(\frac{s}{r}\right)^{N-2} \right] v_{l-1}^\alpha(s) \, ds,
\end{equation*}
 \begin{equation*}
\overline{v}(r) \geq \frac{ C^\beta_{2l-2} }{N-2} \int_\rho^r s \tilde q(s) \left[1 -
    \left(\frac{s}{r}\right)^{N-2} \right] u_{l-1}^\beta(s) \, ds,
\end{equation*}
that is equivalent to (\ref{6}) with $k=l,\,$
$C_{2l} = C^\alpha_{2l-1},\,$ $C_{2l+1} = C^\beta_{2l-2}.$

Let $R \leq s \leq r \leq bR.$ Then $s/r \geq 1/b$ and $(r - s)/s \leq b - 1 \leq 1.$
Using mean value theorem, we obtain
\begin{equation*}
  1 - \left(\frac{s}{r}\right)^{N-2} = \frac{r^{N-2}-s^{N-2}}{r^{N-2}}
  = \frac{(N-2)\xi^{N-3} (r-s)}{r^{N-2}} \geq
  \frac{N-2}{b^{N-2}} \frac{r - s}{s},
\end{equation*}
where $ s \leq \xi \leq r .$ Since $ (r-s)/s \leq 1,$ we conclude that
\begin{gather}
  1 - \left(\frac{s}{r}\right)^{N-2} \geq \frac{N-2}{b^{N-2}}
  \left( \frac{r - s}{s} \right)^m, \quad 1 -
  \left(\frac{s}{r}\right)^{N-2} \geq \frac{N-2}{b^{N-2}} \left(
    \frac{r - s}{s} \right)^n. \label{ineq:main3:krnl}
\end{gather}
Now (\ref{eq:main3:radu}), (\ref{eq:main3:radv}), (\ref{ineq:main3:krnl}) imply the inequalities
\begin{align*}
\overline{u}(r) &\geq c_3 \int_R^r (r-s)^m s^{1-m} \tilde p(s) \overline{v}^\alpha(s) \, ds,
  \\
\overline{v}(r) &\geq c_3 \int_R^r (r-s)^n s^{1-n} \tilde q(s)  \overline{u}^\beta(s) \, ds,
\end{align*}
where $c_3=1/{b^{N-2}},\,$  $R \geq \rho.$ Applying Lemma~\ref{l:I1} with $h = \overline{u},$ $g = \overline{v},$
$P(s) = s^{1-m} \tilde p(s),$ $Q(s) = s^{1-n} \tilde q(s),$ $A = aR,$ we have
\begin{gather}
  y^{\frac{\alpha\beta -1}{(n+1)\alpha+m+1}}(aR)
  \int_{aR}^{bR}\left(s^{1-m}  \tilde p(s) (s^{1-n} \tilde q(s))^\alpha
  \right)^{\frac{1}{(n+1)\alpha+m+1}} \, ds \leq  C_y,
  \label{ineq:main3:limy}
  \\
  z^{\frac{\alpha\beta-1}{(m+1)\beta+n+1}}(aR)
  \int_{aR}^{bR}\left( s^{1-n} \tilde q(s) (s^{1-m}  \tilde p(s))^\beta
  \right)^{\frac{1}{(m+1)\beta +n +1}} \, ds \leq  C_z.
  \label{ineq:main3:limz}
\end{gather}
Combining (\ref{ineq:main3:limy}), (\ref{ineq:main3:limz})  with (\ref{6}), one obtains
\begin{gather*}
\left( C^\alpha_{2k+1}  y_k(R) \right)^{\frac{\alpha\beta -1}{(n+1)\alpha+m+1}}
  \int_{aR}^{bR}\left(s^{1-m}  \tilde p(s) (s^{1-n} \tilde q(s))^\alpha
  \right)^{\frac{1}{(n+1)\alpha+m+1}} \, ds \leq  C_y,     \\
\left( C^\beta_{2k}  z_k(R) \right)^{\frac{\alpha\beta-1}{(m+1)\beta+n+1}}
  \int_{aR}^{bR}\left( s^{1-n} \tilde q(s) (s^{1-m}  \tilde p(s))^\beta
  \right)^{\frac{1}{(m+1)\beta +n +1}} \, ds \leq  C_z,
 \end{gather*}
that contradicts (\ref{eq:main3:oscy}), (\ref{eq:main3:oscz}).

\end{proof}

\begin{remark} \label{Rem1}
Let the conditions of Theorem~\ref{Th2} hold for the problem (\ref{eq:main:task}) with
$p(x) = p_0 (x)$ and $q(x) = q_0 (x).$ Then from the proof of Theorem~\ref{Th2} we conclude that
there are not nontrivial nonnegative entire solutions of (\ref{eq:main:task})
with any $p(x)$ and $q(x)$ satisfying the inequalities
\begin{equation*}
\tilde p (r) \geq \tilde p_0 (r), \,\,\,  \tilde q (r) \geq \tilde q_0 (r).
\end{equation*}
\end{remark}


\section{The proof of Theorem~\ref{Th1}}\label{proof}

\begin{proof}
Theorem~\ref{Th1} is proved in \cite{Y,T1} under the conditions (i) and (ii).
Let us consider the problem (\ref{eq:main:task}) with $p(x)$ and $q(x)$ such that
\begin{equation}\label{23}
\tilde p(r) = \frac{L_1}{r^\lambda \ln^\nu r} ,\quad
\tilde q(r) = \frac{L_2}{r^\mu \ln^\xi r}, \quad r \geq r_0 > 1,
\end{equation}
Let (iii) hold, that is,
\begin{equation}  \label{7}
2 - \mu + \beta (2 - \lambda) = 0, \, \lambda < 2, \, 1 - \xi - \beta \nu > 0.
\end{equation}
Using the principle of mathematical induction, we prove the estimates
\begin{gather}\label{8}
  v_{2k}(r) \geq \bar C_{2k} (\ln r)^{\frac{(\alpha^{k} \beta^{k} - 1)(1
      - \beta \nu - \xi)}{\alpha \beta - 1}},\\
  \label{9}
  u_{2k+1}(r) \geq \bar C_{2k+1} r^{2-\lambda}
  (\ln r)^{\frac{(\alpha^{k} \beta^{k} - 1)(1 - \beta \nu - \xi) \alpha}{\alpha \beta - 1} - \nu},
\end{gather}
where $u_k(r)$ and $v_k(r)$ are defined in (\ref{eq:main3:defu0v0}) --
(\ref{eq:main3:defvk}), $ k =0,1,2,\dots,$  $r > r_k$ for some $r_k >r_\ast = \max (\rho, r_0).$
First, we prove (\ref{8}), (\ref{9}) for $k=0.$ Then (\ref{8}) follows from (\ref{eq:main3:defu0v0}).
Easy to see that
\begin{equation}  \label{10}
1 - \left(\frac{s}{r}\right)^{N-2} \geq  \frac{r - s}{r}   \,\,\, \textrm{for} \,\,\,  s \in (0,r).
\end{equation}
 From  (\ref{eq:main3:defu0v0}), (\ref{eq:main3:defuk}), (\ref{23}),  (\ref{10}), we conclude
\begin{equation*}
  u_1(r) \geq \frac {L_1} r \int_{r_\ast}^r  (r-s) s^{1-\lambda} \ln^{-\nu} s
   \, ds.
\end{equation*}
Using \eqref{7} and the integration by parts, we obtain
\begin{equation*}
  u_1(r) \geq \bar C_1
  r^{2-\lambda} \ln^{-\nu} r,  \;  r > r_1
\end{equation*}
for a suitable choice $\bar C_1$ and $r_1.$
Assume (\ref{8}) and (\ref{9}) are true for $k=l-1.$ In view of
 (\ref{eq:main3:defvk}),  (\ref{23}), (\ref{7}), (\ref{10}) and (\ref{9}) with $k=l-1,$ we have
 \begin{equation*}
 \begin{split}
  v_{2l} & \geq  L_2 \bar C_{2l - 1}^\beta \int_{r_\ast}^r
  s^{1 - \mu + (2 - \lambda) \beta} (\ln s)^{-\xi - \beta \nu +
    \frac{(\alpha^{l - 1} \beta^{l - 1} - 1)  (1 - \xi - \beta \nu ) \alpha \beta}{\alpha \beta - 1}} \left(1 -
    \left(\frac{s}{r}\right)^{N-2}\right) \, ds
    \\
   & \geq
  \frac{L_2 \bar C_{2l - 1}^\beta}{r} \int_{r_\ast}^r \frac{r-s}{s} (\ln s)^{-\xi - \beta \nu +
    \frac{(\alpha^{l - 1} \beta^{l - 1} - 1)  (1 - \xi - \beta \nu ) \alpha \beta}{\alpha \beta - 1}} \, ds.
  \end{split}
\end{equation*}
Integrating by parts on the right side of last inequality, we deduce
(\ref{8}) with $k=l.$
It follows from (\ref{eq:main3:defuk}), (\ref{23}), (\ref{7}), (\ref{10}) and (\ref{8}) with $k=l-1,$ that
 \begin{equation*}
   u_{2 l + 1}(r) \geq \frac{L_1 \bar C_{2 l}^\alpha}r \int_{r_\ast}^r
  ( r - s ) s^{1 - \lambda} ( \ln s )^{-\nu + \frac{( \alpha^{l} \beta^{l}
      - 1 )( 1 - \xi - \beta \nu ) \alpha}{\alpha \beta - 1}} \, ds.
\end{equation*}
Applying the integration by parts, we prove (\ref{9}) with $k=l.$

Now we check (\ref{eq:main3:oscz}). To do it we estimate the multipliers on left side of (\ref{eq:main3:oscz}).
For the convenience, we denote
  \begin{equation*}
 \sigma_k = -\xi - \beta \nu + \frac{(\alpha^k \beta^k - 1)  (1 - \xi - \beta \nu ) \alpha \beta}{\alpha \beta - 1}.
\end{equation*}
It is easy to see that $ \sigma_k > 0$ for large values of $k.$
Applying (\ref{yz}), (\ref{23}), (\ref{7}), (\ref{9}) and mean value theorem, we get
\begin{equation}\label{z}
\begin{split}
  z_{2k+1}(R) & = \int_R^{aR} (aR-s)^n s^{1-n} \tilde q(s) u_{2k+1}^\beta(s) \, ds  \\
  & \geq  L_2 \bar C_{2k + 1}^\beta  \int_R^{aR} (aR-s)^n s^{1-n -\mu + (2 - \lambda) \beta}  \ln^{\sigma_k} s \, ds \\
    & \geq L_2 \bar C_{2k + 1}^\beta \int_R^{(a+1)R/2} \left( \frac{aR}{s} - 1 \right)^n s^{-1}  \ln^{\sigma_k} s  \, ds \\
    & \geq \left( \frac{a-1}{a+1} \right)^n \frac{L_2 \bar C_{2k + 1}^\beta}{\sigma_k + 1}
     \left\{ \ln^{\sigma_k + 1} \left(  \frac{(a+1)R}{2} \right) -  \ln^{\sigma_k + 1} R \right\}  \\
      & \geq \left( \frac{a-1}{a+1} \right)^{n+1} L_2 \bar C_{2k + 1}^\beta \ln^{\sigma_k} R
   \end{split}
\end{equation}
for large values of $R$ and $k.$ Using (\ref{23}), (\ref{7})  and mean value theorem, we find
\begin{equation}\label{int}
   \begin{split}
  & \int_{aR}^{bR} \left( s^{1-n} \tilde q(s) (s^{1-m} \tilde p (s))^\beta   \right)^{\frac{1}{(m+1)\beta+n+1}} \, ds \\
  & = L_2 L_1^\beta  \int_{aR}^{bR} \left( s^{1-n -\mu +(1-m) \beta - \lambda \beta} \ln^{-\xi - \beta \nu} s \right)^{\frac{1}{(m+1)\beta+n+1}} \, ds \\
  & =  L_2 L_1^\beta  \int_{aR}^{bR} s^{-1} (\ln s)^{-\frac{\xi + \beta \nu}{(m+1)\beta+n+1}}  \, ds \\
  & =  L_2 L_1^\beta \frac{(m+1)\beta+n+1}{(m+1)\beta+n+1 -\xi - \beta \nu}  \left[ (\ln (bR))^{\frac{(m+1)\beta+n+1 -\xi - \beta \nu}{(m+1)\beta+n+1}} -
  (\ln (aR))^{\frac{(m+1)\beta+n+1 -\xi - \beta \nu}{(m+1)\beta+n+1}} \right] \\
  & =  L_2  L_1^\beta \frac{b-a}{\gamma}  (\ln (\gamma R))^{-\frac{\xi +\beta \nu}{(m+1)\beta+n+1}},
 \end{split}
\end{equation}
where $R > r_0,$  $\gamma \in (a,b).$ Now from (\ref{z}), (\ref{int}), we conclude
\begin{equation*}
  \limsup_{R \to \infty}
  z_{2k+1}^{\frac{\alpha\beta-1}{(m+1)\beta+n+1}}(R)
  \int_{aR}^{bR} \left( s^{1-n} \tilde q(s) (s^{1-m} \tilde p (s))^\beta
  \right)^{\frac{1}{(m+1)\beta+n+1}} \, ds
\end{equation*}
\begin{equation*}
\geq  \hat C_{k} \lim_{R\to\infty}
  (\ln R)^{\frac{\alpha\beta}{(m+1)\beta+n+1}\left(\alpha^k\beta^k (1  - \xi - \beta\nu) - 1 \right)} = \infty
\end{equation*}
for large values of $k.$ Since $\lambda < 2$ and $\mu > 2$ the functions $r^{1-n} \tilde q (r)$ and $r^{1-m} \tilde p (r)$
are nonincreasing for any $n \in \mathbf{N},$  $m > 1 - \lambda$ and large $r.$
Thus, by Theorem~\ref{Th2} and Remark~\ref{Rem1} there are not nontrivial
nonnegative entire solutions of (\ref{eq:main:task}).

The case (iv) is treated in a similar way.

We note that (v) follows from (vi) if $\nu < 1$ and $\xi \geq  1,$
(vi) follows from (v) if $\nu \geq 1$ and $\xi < 1,$ and
(v) and (vi) are equivalent if $\nu < 1$ and $\xi < 1.$
So, we can prove the theorem for (v) under the condition $\nu < 1$ and for (vi) under the condition  $\xi  < 1.$

Let (v) and $\nu < 1$ hold, that is,
\begin{equation}\label{20}
\lambda = 2,\, \mu = 2, \, 1 - \xi + \beta (1-\nu) > 0, \, \nu < 1.
\end{equation}
As in a previous case, using the principle of mathematical induction, we prove the estimates
\begin{equation}  \label{ineq:cons3:case3-v2k}
  v_{2k}(r) \geq \breve C_{2 k}
  (\ln r)^{\frac{(\alpha^k\beta^k-1)(1-\xi+\beta(1-\nu))}{\alpha\beta-1}},
\end{equation}
\begin{equation*}
    u_{2k+1}(r)\geq \breve C_{2 k + 1}
  \left( \ln r \right)^{\frac{(\alpha^k\beta^k-1)(1-\xi+\beta(1-\nu))\alpha}
    {\alpha\beta-1}+1-\nu},
\end{equation*}
where $ k =0,1,2,\dots,$  $r > \bar r_k$ for some $\bar r_k > r_\ast.$ By virtue of
(\ref{yz}), (\ref{23}), (\ref{20}), (\ref{ineq:cons3:case3-v2k}), we have
\begin{equation*}
   \limsup_{R \to \infty} y_k^{\frac{\alpha\beta -
      1}{(n+1)\alpha+m+1}}(R)
  \int_{aR}^{bR} \left( s^{1-m} \tilde p (s) (s^{1-n} \tilde q(s))^\alpha
  \right)^{\frac{1}{(n+1)\alpha+m+1}} \, ds
\end{equation*}
\begin{equation*}
\geq  \tilde C_k \lim_{R\to\infty} (\ln R)^{\frac{\alpha^{k + 1}
      \beta^k (1 - \xi + \beta (1 - \nu)) - \alpha - \alpha
      \beta}{(n+1)\alpha+m+1}} = \infty
\end{equation*}
for large values of $k.$ Obviously,  $r^{1-m} \tilde p (r)$ and $r^{1-n} \tilde q (r)$
are nonincreasing  functions for any $m,n \in \mathbf{N}$ and large $r.$
Applying Theorem~\ref{Th2} and Remark~\ref{Rem1} again, we prove that there are not nontrivial
nonnegative entire solutions of (\ref{eq:main:task}).

If (vi) and  $\xi < 1$ hold then the theorem is proved in a very similar manner.
\end{proof}

\begin{remark}
Let $ p $ and $ q $ have spherical symmetry and $(u, v)$ be a nonnegative spherically symmetric entire solution of (\ref{eq:main:task}).
If $\alpha > 0,$ $\beta > 0$ then $(u, v)$ satisfies the following problem
\begin{equation*}
(r^{N-1} u'(r))' =  r^{N-1}  p (r) v^\alpha (r), \, r>0, \quad u'(0) = 0,
\end{equation*}
\begin{equation*}  \
(r^{N-1} v'(r))' =  r^{N-1}  q (r) u^\beta (r), \, r>0, \quad v'(0) = 0.
\end{equation*}
We note that the conditions $\alpha \geq 1,$ $\beta \geq 1$ of Theorem~\ref{Th1} and Theorem~\ref{Th2} are used for (\ref{4}) and (\ref{5}) only.
Hence we can state in Theorem~\ref{Th1} and Theorem~\ref{Th2} the nonexistence of nontrivial nonnegative spherically symmetric entire solutions of (\ref{eq:main:task}) without the assumptions $\alpha \geq 1,$ $\beta \geq 1.$
\end{remark}


\section{The optimality of Theorem~\ref{Th1}}
In this section, we show the optimality of Theorem~\ref{Th1}. We assume that the system (\ref{eq:main:task}) has spherically symmetric coefficients
$ p (x) =  p (|x|), \, $  $ q (x) =  q (|x|)$ which satisfy the inequalities
\begin{equation}\label{15}
p(r) \leq \frac{L_3}{r^\lambda \ln^\nu r} ,\quad
q(r) \leq \frac{L_4}{r^\mu \ln^\xi r}, \quad r \geq r_1 > 1,
\end{equation}
where $L_3 > 0,\, $ $L_4 > 0.$

The following statement is proved by a direct verification of the conditions (\ref{Lair-2011:exist-Q}) and  (\ref{Lair-2011:exist-P}).
\begin{theorem}\label{Th3}
Let $N \geq 3,$ $\alpha \beta > 1$ and $p$ and $q$ satisfy (\ref{15}). If one from the following conditions holds
       \begin{equation*}
 \begin{split}
 & \textrm{(i)} \,\,\,\,\,\,\,    2 - \mu+  \beta (2 - \lambda)  < 0, \;  2 - \lambda + \alpha (2 - \mu) < 0; \\
 & \textrm{(ii)} \,\,\,\,\,       2 - \mu+  \beta (2 - \lambda) = 0,  \;  \lambda < 2,  \;  1 - \xi - \beta \nu < 0;  \\
 & \textrm{(iii)} \,\,\,          2 - \lambda + \alpha (2 - \mu) = 0,  \;  \mu < 2,  \;  1 - \nu - \alpha \xi  <0;   \\
 & \textrm{(iv)} \,\,\,\,         \lambda = 2, \;  \mu = 2,  \;  1 - \xi + \beta (1-\nu) < 0,  \;  1-\nu+\alpha(1-\xi) <0,
 \end{split}
\end{equation*}
then (\ref{eq:main:task}) has positive entire solutions.
\end{theorem}

Let the conditions of Theorem~\ref{Th1} and Theorem~\ref{Th3} hold with $L_3 \geq L_1, $ $L_4 \geq L_2.$
The figures 1--4 show values of the parameters $\lambda,$ $\mu,$ $\nu$ and $\xi$ in (\ref{3}) and (\ref{15})
providing the nonexistence and existence of nontrivial nonnegative entire solutions of (\ref{eq:main:task}).

\begin{figure}[h!]
\begin{minipage}[h]{0.47\linewidth}
\center{\includegraphics{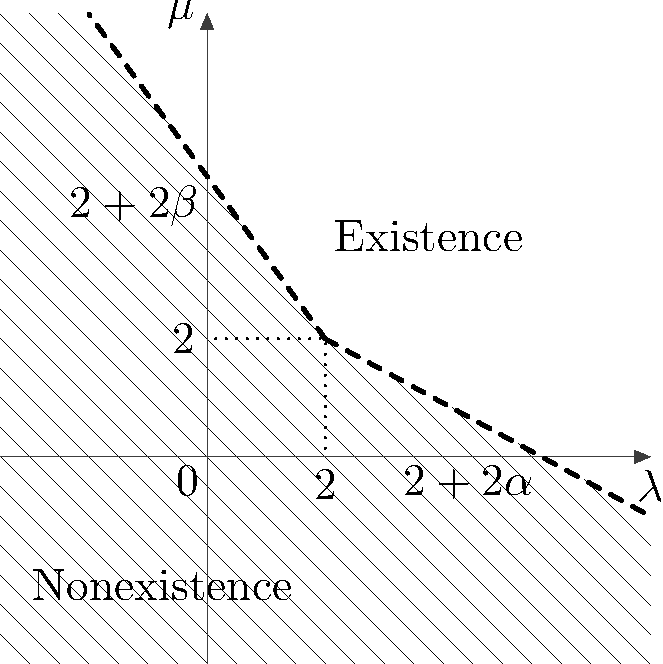}}\hspace{25pt}
\caption{$\nu$ and $\xi$ are any.}
\end{minipage}
\hfill
\begin{minipage}[h]{0.47\linewidth}
\center{\includegraphics{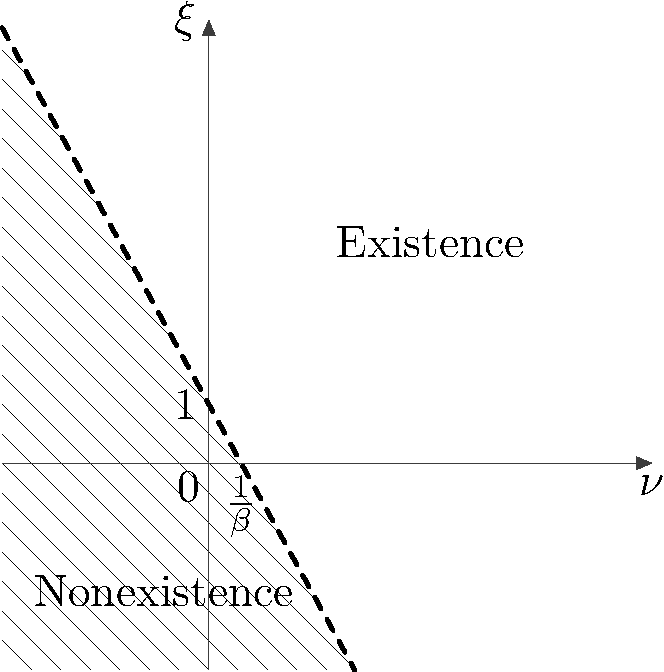}}
\caption{$2 - \mu+  \beta (2 - \lambda) = 0$ and $\lambda < 2.$}
\end{minipage}
\end{figure}
\newpage
\begin{figure}[h!]
\begin{minipage}[h]{0.47\linewidth}
\center{\includegraphics{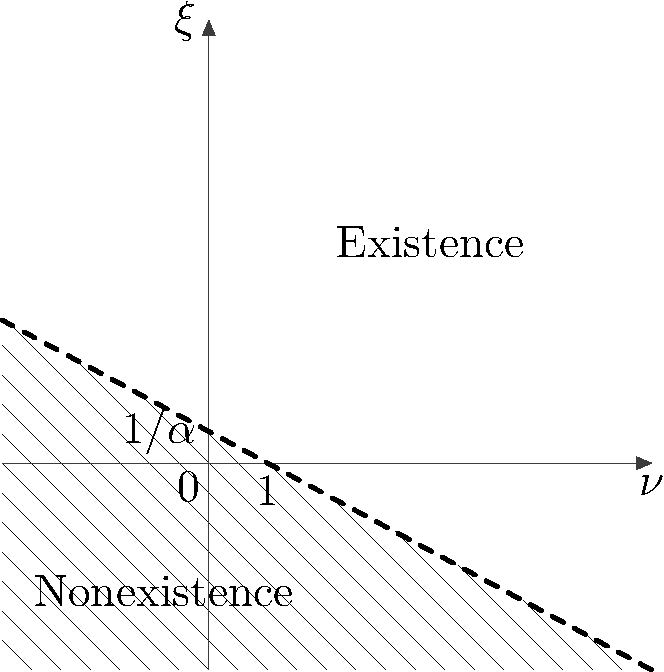}}\hspace{25pt}
\caption{$ 2 - \lambda + \alpha (2 - \mu) = 0$ and $\mu < 2.$ }
\end{minipage}
\hfill
\begin{minipage}[h]{0.47\linewidth}
\center{\includegraphics{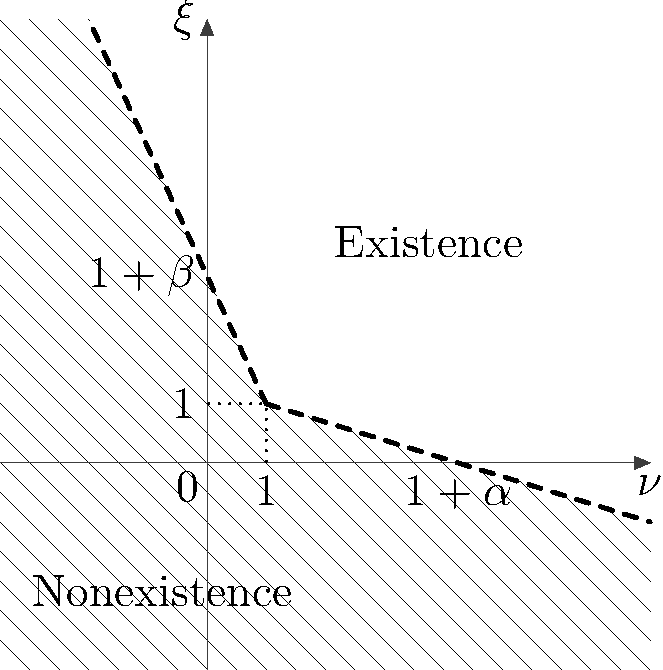}}
\caption{$\lambda = 2$ and $\mu = 2.$}
\end{minipage}
\end{figure}

\section*{Disclosure statement}

No potential conflict of interest was reported by the authors.

\section*{Funding}

A. Gladkov is supported by the "RUDN University Program 5-100" and the state program of fundamental research of Belarus
(grant 1.2.03.1).

\end{document}